\documentclass[reqno, 12pt]{amsart}
\usepackage{amsmath}
\usepackage{amssymb}
\usepackage{amsthm}
\usepackage{mathtools}
\usepackage{tikz-cd}
\usepackage{enumerate}
\usepackage[mathscr]{eucal}
\usepackage{graphicx}

\newtheorem{theorem}{Theorem}[section]
\newtheorem{lemma}[theorem]{Lemma}
\newtheorem{proposition}[theorem]{Proposition}
\newtheorem{corollary}[theorem]{Corollary}
\newtheorem{definition}[theorem]{Definition}
\newtheorem{definitions}[theorem]{Definitions}
\theoremstyle{definition}
\newtheorem{example}[theorem]{Example}
\newtheorem{remark}[theorem]{Remark}

\begin{document}
	\title[]{Barreled extended locally convex spaces and Uniform boundedness principle}
	\author{Akshay Kumar \and Varun Jindal}
	\address{Akshay Kumar: Department of Mathematics, Malaviya National Institute of Technology Jaipur, Jaipur-302017, Rajasthan, India}
	\email{akshayjkm01@gmail.com}
		
	\address{Varun Jindal: Department of Mathematics, Malaviya National Institute of Technology Jaipur, Jaipur-302017, Rajasthan, India}
	\email{vjindal.maths@mnit.ac.in}

	\maketitle
\begin{abstract}	
	For an extended locally convex space $(X,\tau)$, in \cite{flctopology}, the authors studied the finest locally convex topology (flc topology) $\tau_F$ on $X$ coarser than $\tau$. One can often prove facts about $(X, \tau)$ by applying classical locally convex space theory on $(X, \tau_F)$.  This paper employs the flc topology to analyze barreled extended locally convex spaces and establish the uniform boundedness principle in the extended setting. One of the key results of this paper is the relationship between the barreledness of an extended locally convex space $(X,\tau)$ and the barreledness of the associated finest locally convex space $(X, \tau_F)$.  This is achieved by examining the lower semi-continuous seminorms on these spaces.

\end{abstract}
\section{Introduction}

Beer \cite{nwiv} initiated the study of extended normed linear spaces. Beer and Vanderwerff further developed these spaces in \cite{socsiens, spoens}.  Motivated by the work of Beer and Vanderwerff, Salas and Tapia-Garc{\'\i}a introduced the concept of an extended locally convex space (elcs) in \cite{esaetvs}.  
The structure of these new spaces differs significantly from that of classical locally convex spaces. In \cite{flctopology}, authors studied the finest locally convex topology (flc topology, for short) $\tau_{F}$ for an extended locally convex space $(X, \tau)$ that is still coarser than $\tau$. It is shown that both $(X, \tau)$ and $(X, \tau_F)$ have the same collection of continuous seminorms. In particular, both have the same dual. We call the space $(X,\tau_F)$ the finest locally convex space associated with the elcs $(X,\tau)$. 

In this paper, we define and investigate barreled extended locally convex spaces. The associated finest locally convex space plays a significant role in this investigation. To find the relationship between the barreledness of an elcs $(X,\tau)$ and the barreledness of the corresponding finest space $(X, \tau_F)$, we first examine lower semi-continuous seminorms on these spaces. Finally, by using the relationship between the barreledness of $(X,\tau)$ and $(X, \tau_F)$, we prove the uniform boundedness principle for the elcs $(X,\tau)$.


 
The entire work of this paper is presented in five sections. Section 2 provides preliminary definitions and features of extended locally convex spaces. In the third section, we study barreled extended locally convex spaces. We begin our investigation by characterizing the barreled subspaces of an elcs. Then as a main result, we show that an elcs $(X,\tau)$ is barreled if and only if the finest space $(X,\tau_F)$ is barreled. We use this result to prove some permanence properties of a barreled elcs.


In the fourth section, we look at the barreledness of the space $C(X)$ endowed with the topology $\tau_{\mathcal{B}}^s$ (or $\tau_{\mathcal{B}}$) of strong uniform convergence (or uniform convergence) on the elements of a bornology $\mathcal{B}$. We show that the bornology $\mathcal{K}$ of all relatively compact sets is coarser than $\mathcal{B}$ if $(C(X), \tau_{\mathcal{B}}^s)$ is barreled.  In the final section, we examine equicontinuous families of continuous linear operators between two extended locally convex spaces, and prove the uniform boundedness principle in the context of an elcs.

\section{Preliminaries}
Throughout the paper, we assume that the underlying field for a vector space is $\mathbb{K}$ which is either $\mathbb{R}$ or $\mathbb{C}$. We also adopt the following conventions for the $\infty$: $\infty.0=0.\infty=0$; $\infty+\alpha=\alpha+\infty=\infty$ for every $\alpha\in\mathbb{R}$; $\infty.\alpha=\alpha.\infty=\infty$ for $\alpha>0$; $\inf\{\emptyset\}=\infty$.  

An extended seminorm $\rho:X\to[0, \infty]$ on a vector space $X$ is a function which satisfies the following properties.
\begin{itemize}
	\item[(1)] $\rho(\alpha x)=|\alpha|\rho(x)$ for each scalar $\alpha$ and $x\in X$;
	\item[(2)] $\rho(x+y)\leq \rho(x)+\rho(y)$ for $x,y\in X$.\end{itemize}

An extended seminorm $\parallel\cdot\parallel:X\to[0, \infty]$ is  said to be an extended norm if $\parallel x\parallel=0$ implies $x=0$. A vector space $X$ together with an extended norm $\parallel\cdot\parallel$ is called an \textit{extended normed linear space} (enls, for short), and it is denoted by $(X, \parallel\cdot\parallel)$. In this situation, $(X, \parallel\cdot\parallel)$ is said to be an \textit{extended Banach space} if it is complete with respect to the metric $d(x, y)= \min\{\parallel x-y\parallel, 1\}$ for $x, y\in X$. It is easy to prove that an enls $(X, \parallel\cdot\parallel)$ is an extended Banach space if and only if the \textit{finite subspace} $(X_{fin}, \parallel\cdot\parallel)$ is a Banach space, where $X_{fin}=\{x\in X:\parallel x\parallel<\infty\}$. For more details about extended normed linear spaces, we refer to \cite{nwiv, socsiens, spoens}.  

A vector space $X$ endowed with a Hausdorff topology $\tau$ is said to be an \textit{extended locally convex space} (elcs, for short) if $\tau$ is induced by a collection $\mathcal{P}=\{\rho_i:i\in\mathcal{I}\}$ of extended seminorms on $X$, that is, $\tau$ is the smallest topology on $X$ under which each $\rho_i$ is continuous. The \textit{finite subspace} $X_{fin}$ of an elcs $(X, \tau)$ is defined by $$X_{fin}=\{x\in X:\rho(x)<\infty \text{ for every continuous extended seminorm } \rho \text{ on } X\},$$ and we define  $X_{fin}^\rho=\{x\in X:\rho(x)<\infty\}$ for any extended seminorm $\rho$ on $X$.  

Suppose $(X, \tau)$ is an elcs and $\tau$ is induced by a family $\mathcal{P}$ of extended seminorms on $X$. Then the following facts are either easy to verify or given in \cite{esaetvs}. 
\begin{itemize}
	\item[(1)] There exists a neighborhood base $\mathcal{B}$ of $0$ in $(X, \tau)$ such that each element of $\mathcal{B}$ is absolutely convex (balanced and convex);
	\item[(2)] $X_{fin}$ with the subspace topology is a locally convex space;
	\item[(3)] $X_{fin}=\bigcap_{\rho\in\mathcal{P}} X_{fin}^\rho;$
	\item[(4)] if $\rho$ is any continuous extended seminorm on $(X, \tau)$, then $X_{fin}^\rho$ is a clopen subspace of $(X, \tau)$;
	\item[(5)] a subspace $Y$ is open in $(X, \tau)$ if and only if there exists a continuous extended seminorm $\rho$ on $(X, \tau)$ such that $X_{fin}^\rho\subseteq Y$. 
	\item[(6)]  $X_{fin}$ is an open subspace in $(X, \tau)$ if and only if there exists a continuous extended seminorm $\rho$ on $(X, \tau)$ such that $X_{fin}=X_{fin}^\rho$. In this case, we say $(X, \tau)$ is a \textit{fundamental elcs}. \end{itemize}

\begin{definition}\label{Minkowski funcitonal}	{\normalfont(\cite{esaetvs})}	
	\normalfont	Suppose $U$ is any absolutely convex set in an elcs $(X, \tau)$. Then the \textit{Minkowski functional} $\rho_U:X\rightarrow[0,\infty]$  for $U$  is defined as  $$\rho_U(x)=\inf\{\lambda>0: x\in \lambda U\}.$$
\end{definition}

The following facts are immediate from the above definition.
\begin{enumerate}
	\item The Minkowski functional $\rho_{U}$  for the set $U$ is an extended seminorm on $X$. In addition, if  $U$ is absorbing, then $\rho_U$ is a seminorm on $X$.
	\item  $\{x\in X:\rho_U(x)<1\}\subseteq U\subseteq \{x\in X:\rho_U(x)\leq 1\}$.
	\item  The Minkowski functional $\rho_U$ is  continuous on $X$ if and only if $U$ is a neighborhood of $0$.
\end{enumerate}

For an absolutely convex subset $U$ of an elcs $X$, we define  $X_{fin}^U=\{x\in X: \rho_U(x)<\infty\}$. It is shown in \cite{esaetvs} that if $\mathcal{B}$ is a neighborhood base of $0$ in an elcs $(X, \tau)$ consisting of absolutely convex sets, then $\tau$ is induced by the collection $\{\rho_{U}: U\in\mathcal{B}\}$. 

If $A$ is any nonempty set in a topological space $(X,\tau)$, then we denote the closure and interior of $A$ in $(X,\tau)$ by $\text{Cl}_\tau(A)$  and  $\text{int}_\tau(A)$, respectively. For other terms and definitions, we refer to \cite{lcsosborne, tvsschaefer, willard}.

\section{Barreled extended locally convex spaces}

We define and explore barreled extended locally convex spaces in this section. On an elcs, we also look into barreled subspaces and lower semicontinuous seminorms. We demonstrate, among other things, that an elcs $(X, \tau)$ is barreled if and only if $(X,\tau_F)$ is barreled if and only if every lower semi-continuous seminorm on $(X, \tau)$ is continuous.  At the end of this section, we examine some permanence properties of a barreled elcs.

\begin{definitions}\label{definition of barreled space}\normalfont A \textit{barrel} $B$ in an elcs $(X, \tau)$ is a closed, absolutely convex and absorbing set in $(X,\tau)$. We say $(X, \tau)$ is \textit{barreled} if every barrel in it is a neighborhood of $0$.\end{definitions}

We first give an example of a barreled elcs. 

\begin{example}\label{barreledness of a discrete space} Let $X$ be a vector space together with the discrete extended norm\[ \parallel x\parallel_{0, \infty}=
	\begin{cases}
	\text{0;} &\quad\text{if $x=0 $}\\
	\text{$\infty$;} &\quad\text{if $x\neq 0.$ }\\
	\end{cases}\] Then every subset of $X$ is open in the enls  $(X, \parallel\cdot\parallel_{0, \infty})$. Therefore every barrel in $(X, \parallel\cdot\parallel_{0, \infty})$ is a neighborhood of $0$. Hence $(X, \parallel\cdot\parallel_{0, \infty})$ is barreled.\end{example}


\begin{theorem}\label{barreledness of a baire space} Suppose $(X,\tau)$ is an elcs which is a Baire space. Then $(X, \tau)$ is barreled.\end{theorem}

\begin{proof}Let $B$ be a barrel in $(X, \tau)$. Then $\displaystyle{X=\cup_{n\in\mathbb{N}}}nB$ as $B$ is absorbing in $X$. Since $X$ is a Baire space, $\text{ int}(n_0B)\neq \emptyset$ for some $n_0\in\mathbb{N}$. Let $x\in\text{ int }(n_0B)$. Then there exists a neighborhood $U$ of $x$ such that $U\subseteq n_0B$. Note that $\frac{(-x)}{2n_0}+\frac{U}{2n_0}$ is a neighborhood of $0$ and $$\frac{(-x)}{2n_0}+\frac{U}{2n_0}\subseteq B.$$  Thus $B$ is a neighborhood of $0$. Which completes the proof.  \end{proof}

\begin{corollary}\label{barreledness of an extended Banach space} If $(X,\parallel\cdot\parallel)$ is an extended Banach space, then $(X,\parallel\cdot\parallel)$ is barreled.\end{corollary}

\begin{corollary}\label{barreledness of a uniform space} Let $(X, d)$ be a metric space. Then  the collection $C(X)$ of all real-valued continuous functions on $X$ endowed with the extended norm $\parallel f\parallel_\infty=\sup_{x\in X}|f(x)|$ for $f\in C(X)$ is a barreled enls.\end{corollary}
\begin{proof} It follows from the fact that $(C(X), \parallel \cdot\parallel_\infty)$ is an extended Banach space.\end{proof}

We say a subspace $Y$ of an elcs $(X, \tau)$ is barreled if $(Y, \tau|_Y)$ is barreled. To study the barreled subspaces of an elcs,  we need the following lemma. 

\begin{lemma}\label{absolutely convex sets in an elcs} Let $(X,\tau)$ be an elcs. Then the following properties hold for $A\subseteq X$. 
\begin{itemize}
	\item[(1)] If $A$ is convex, then $\text{Cl}_\tau(A)$ is convex.
	\item[(2)] If $A$ is balanced, then $\text{Cl}_\tau(A)$ is balanced.\end{itemize}\end{lemma}

\begin{proof} It is similar to the proofs of Proposition 2.13, p. 42, and Proposition 2.5(a), p. 37 in \cite{lcsosborne}.\end{proof}

\begin{proposition}\label{barrel in finite subspace of M} Let $M$ be a subspace of an elcs $(X, \tau)$. Suppose $\rho$ is a continuous extended seminorm on $(X, \tau)$. Then for every barrel  $B$ in $M\cap X_{fin}^\rho$, there exists a barrel $G$ in $M$ such that $G\cap\left( M\cap X_{fin}^\rho\right)=B$.\end{proposition}

\begin{proof} Suppose $Z$ is a subspace of $M$ such that $M=\left( M \cap X_{fin}^\rho\right) \oplus Z$. Let $B$ be a barrel in $M\cap X_{fin}^\rho$. Let $G= \text{Cl}_{\tau|_M}(B+Z)$. By Lemma \ref{absolutely convex sets in an elcs}, $G$ is absolutely convex. It is easy to see that if $x=x_1+x_2\in M$ for $x_1\in M \cap X_{fin}^\rho$ and $x_2\in Z$, then there exists a $c>0$ such that $c(x_1+x_2)\in B+Z$. Therefore $G$ is absorbing in $M$. Consequently, $G$ is a barrel in $M$.  Also, $B\subseteq G\cap \left( M\cap X_{fin}^\rho \right)$. For the reverse inclusion, let $x_0\in G\cap \left( M \cap X_{fin}^\rho\right)$. If $U$ is a neighborhood of $x_0$ in $M \cap X_{fin}^\rho$, then  $U$ is a  neighborhood of $x_0$ in $M$ as $M \cap X_{fin}^\rho$ is an open subspace of $M$. Therefore there is an $x\in M$ such that $x\in U\cap \left(B+Z \right)$. Consequently, $x\in U\cap B$. Since $B$ is closed in $M \cap X_{fin}^\rho$, $x_0\in B$. Hence $G\cap\left( M\cap X_{fin}^\rho\right)=B$.
\end{proof}	

%
%
%

\begin{corollary} Suppose $\rho$ is a continuous extended seminorm on an elcs $(X, \tau)$. Then for every barrel $B$ in $X_{fin}^\rho$, there exists a barrel $G$ in $X$ such that $G\cap\left(X_{fin}^\rho\right)=B$.
\end{corollary}
\begin{theorem}\label{barreledness of finite subspace of M} Suppose $M$ is a subspace of an elcs $(X, \tau)$. Then the following statements are equivalent.	
\begin{itemize}
\item[(1)] $M$ is a barreled subspace of $X$;
\item[(2)] $M\cap X_{fin}^\rho$ is a barreled subspace of $X$ for every continuous extended seminorm $\rho$ on $X$;
\item[(3)] $M\cap X_{fin}^\rho$ is a barreled subspace of $X$ for some continuous extended seminorm $\rho$ on $X$.	\end{itemize}\end{theorem}

\begin{proof} $(1)\Rightarrow (2)$. Suppose $\rho$ is a continuous extended seminorm on $(X, \tau)$ and $B$ is a barrel in  $M\cap X_{fin}^\rho$. By Proposition \ref{barrel in finite subspace of M}, there exists a barrel $G$ in $M$ such that $G\cap \left( M\cap X_{fin}^\rho\right) =B$. Since $M$ is a barreled subspace of $X$, $G$ is a neighborhood of $0$ in $M$. Therefore $B=G\cap\left( M\cap X_{fin}^\rho\right)$ is  a neighborhood of $0$ in $M\cap X_{fin}^\rho$. Hence $M\cap X_{fin}^\rho$ is a barreled subspace of $X$.
	
$(2)\Rightarrow (3)$. It is immediate.
	
$(3)\Rightarrow(1)$. Let $G$ be a barrel in $M$. Then $G\cap \left( M\cap X_{fin}^\rho\right)$ is closed  in $M\cap X_{fin}^\rho$. Note that $G\cap \left( M\cap X_{fin}^\rho\right)$ is absolutely convex and absorbing in $M\cap X_{fin}^\rho$. Therefore $G\cap \left( M\cap X_{fin}^\rho\right)$ is a barrel in $M\cap X_{fin}^\rho$. As $M\cap X_{fin}^\rho$ is barreled,  $G\cap \left( M\cap X_{fin}^\rho\right)$ is a neighborhood of $0$ in $M\cap X_{fin}^\rho$. Consequently, $G$ is a neighborhood of $0$ in $M$. Which completes the proof.\end{proof}

\begin{corollary}\label{barreledness of finite subspace of an elcs X}	Suppose $(X,\tau)$ is an elcs. Then the following statements are equivalent.
\begin{itemize}
\item[(1)] $X$ is barreled;
\item[(2)] $X_{fin}^\rho$ is a barreled  subspace of $X$ for every continuous extended seminorm $\rho$ on $X$;
\item[(3)] $X_{fin}^\rho$ is a barreled  subspace of $X$ for some continuous extended seminorm $\rho$ on $X$.	\end{itemize}\end{corollary}
	
\begin{corollary}\label{barreledness of finite subspace of an enls X} Suppose $(X, \parallel\cdot\parallel)$ is an enls. Then $X$ is barreled if and only if $(X_{fin}, \parallel\cdot\parallel)$ is barreled.\end{corollary}

In general, we may not replace $X_{fin}^\rho$ by $X_{fin}$ in Theorem \ref{barreledness of finite subspace of M}.

\begin{example}	Let $X=c_{00}$ be the collection of all real sequences which are eventually zero. For $n\in\mathbb{N}$, consider the extended norm $\rho_n:X\rightarrow [0,\infty]$ defined by 	\[\rho_n((x^m))=
	\begin{cases}
		\text{$\infty$,} &\quad\text{ if $x^m\neq 0$  for some $1\leq m\leq n$}\\
		\text{$\displaystyle{\sup_{m\in\mathbb{N}}}|x^m|$,} &\quad\text{if $x^m= 0$ for every $1\leq m\leq n$.}\\
	\end{cases}
	\] Let $\tau$ be the topology induced by the collection $\mathcal{P}=\{\rho_n:n\in\mathbb{N}\}$. Then $(X, \tau)$ is an elcs and $X_{fin}=\{0\}$. Therefore $X_{fin}$ is a barreled subspace of $(X, \tau)$. Now, suppose $B=\{(x^m)\in X: |x^m|\leq \frac{1}{m} \text{ for } m\in\mathbb{N} \}$. Then $B$ is absolutely convex. If $x=(x^m)$ is any element in $X$, then there exists an $m_0\in\mathbb{N}$ such that $x^m=0$ for $m\geq m_0$. Consequently, $c.x\in B$ for $c= \frac{1}{m_0 \max\{|x^m|+1: m\in\mathbb{N}\}}$. Thus $B$ is absorbing in $X$. Let $(z_n)$ be a sequence in $B$ such that $z_n=(z_n^m)$ for each $n\in\mathbb{N}$. If $z_n\to z=(z^m)$ in $X$, then for every $m\in\mathbb{N}$, $z_n^m\to z^m$ in $\mathbb{R}$. Therefore $|z^m|\leq \frac{1}{m}$ for every $m\in\mathbb{N}$. Consequently, $z\in B$. Note that $(X, \tau)$ is a metrizable elcs as $\tau$ is induced by a countable family of elcs (see, Lemma 4.15 in \cite{esaetvs}). Thus $B$ is closed in $X$. Hence $B$ is a barrel in $X$. If $B$ is a neighborhood of $0$ in $(X, \tau)$, then $\{x\in X: \rho_{n_0}(x)<\epsilon\}\subseteq B$ for some  $n_0\in\mathbb{N}$ and $\epsilon>0$. Suppose $k\in\mathbb{N}$ such that $\frac{1}{k}<\min\{\frac{1}{n_0}, ~\frac{\epsilon}{2}\}$. Take $y=(y^m)$ defined by
	\[y^m=
	\begin{cases}
		\text{$\frac{\epsilon}{2}$;} &\quad\text{ for $m=k$}\\
		\text{0;} &\quad\text{ otherwise. }\\
	\end{cases}\] Then $y\in\{x\in X: \rho_{n_0}(x)<\epsilon\}\setminus B$. Hence $(X, \tau)$ is not barreled.  \end{example}	

Our next aim is to relate the barreledness of an elcs $(X,\tau)$ with the barreledness of the finest space $(X,\tau_F)$. To do that, first we study lower semi-continuous seminorms on $(X,\tau)$. 

\begin{definitions}{\normalfont(\cite{willard})} \normalfont A function $f:(X, \tau)\to \mathbb{R}$ on a topological space $(X, \tau)$ is said to be  \textit{lower semi-continuous at $x_0\in X$} if for every $\epsilon>0$, $f^{-1}((f(x_0)-\epsilon, \infty))$ is a neighborhood of $x_0$ in $(X, \tau)$. The function $f$ is  said to be \textit{lower semi-continuous on $X$} if it is lower semi-continuous at every point of $X$.\end{definitions}

Note that a function $f:(X, \tau)\to \mathbb{R}$ on a topological space $(X, \tau)$ is lower semi-continuous on $X$ if and only if $f^{-1}((c, \infty))\in \tau$ for every $c\in \mathbb{R}$ if and only if $f^{-1}((-\infty, c])$ is closed in $(X, \tau)$ for every $c\in \mathbb{R}$.   

In \cite{flctopology}, it is shown that if $\tau_F$ is the flc topology for an elcs $(X,\tau)$, then both $(X,\tau)$ and $(X,\tau_F)$ have the same collection of continuous seminorms. Our next result shows that this is even true for lower semi-continuous seminorms.   
\begin{proposition}\label{lower semicotinuous seminorms} Let $P:X\rightarrow [0,\infty)$ be a seminorm on an elcs $(X, \tau)$. If $\tau_F$ is the flc topology for $(X, \tau)$, then $P$ is lower semi-continuous with respect to $\tau$ if and only if $P$ is lower semi-continuous with respect to $\tau_F$.\end{proposition}

\begin{proof} Suppose $P$ is lower semi-continuous  with respect to $\tau$. Then for $x_0\in X$ and $\epsilon>0$, there exists a continuous extended seminorm $\rho$ and $c>0$ such that $$P\left(x_0+\{x\in X:\rho(x)<c\}\right) \subseteq \left( P(x_0)-\frac{\epsilon}{2}, \infty\right).$$ 
	
Let $M$ be a  subspace of $X$ such that $X=X_{fin}^\rho\oplus M$. Consider the seminorm $\mu:X=X_{fin}^\rho\oplus M\rightarrow [0,\infty)$ defined by $$\mu(x)=\rho(x_f)+P(x_m),$$ where $x=x_f+x_m$ for $x_f\in X_{fin}^\rho$ and $x_m\in M$. Since $\mu(x)\leq \rho(x)$ for every $x\in X$, $\mu$ is  a continuous seminorm on $(X,\tau)$. By Theorem 3.5 in \cite{flctopology}, $\mu$ is continuous with respect to $\tau_F$. It is easy to prove that $$P\left(x_0+\left\lbrace x\in X:\mu(x)<\min\left\lbrace c, \frac{\epsilon}{2}\right\rbrace \right\rbrace \right)\subseteq \left( P(x_0)-\epsilon, \infty\right).$$ Therefore $P$ is lower semi-continuous at $x_0$ with respect to $\tau_F$. Consequently, $P$ is lower semi-continuous on $X$ with respect to $\tau_F$.\end{proof}

\begin{theorem}\label{barreledness of flc topology} Let $(X,\tau)$ be an elcs. If $\tau_F$ is the flc topology for $(X, \tau)$, then the following statements are equivalent.
\begin{itemize}
\item[(1)] $(X,\tau)$ is barreled.
\item[(2)] $(X,\tau_{F})$ is barreled.
\item[(3)]Every lower semi-continuous seminorm on $(X,\tau_F)$ is continuous.
\item[(4)]Every lower semi-continuous seminorm on $(X, \tau)$ is continuous. \end{itemize} \end{theorem}

\begin{proof}	$(1)\Rightarrow (2).$ Suppose $B$ is a barrel in $(X,\tau_{F})$. Then  $B$ is a barrel in $(X,\tau)$ as $\tau_F\subseteq \tau$. Since $(X,\tau)$ is barreled, $B$ is a neighborhood of $0$ in $(X,\tau)$. Then the Minkowski functional $\rho_B$ is a continuous seminorm on $(X,\tau)$. By Theorem 3.5 in \cite{flctopology}, $\rho_B$ is continuous with respect to  $\tau_{F}$. Hence $B$ is a neighborhood of $0$  in $(X,\tau_{F})$ as $\rho_B^{-1}([0, 1))\subseteq B$.
	
$(2)\Rightarrow(3)$. It follows from Theorem 11.4.3, p. 386 in \cite{tvsnarici}.
	
$(3)\Rightarrow(4)$. It follows from Proposition \ref{lower semicotinuous seminorms}. 

$(4)\Rightarrow(1)$. Suppose $B$ is a barrel in $(X,\tau)$. Then the Minkowski functional $\rho_B$ is a seminorm on $X$. For any $c> 0$, let $(x_\lambda)_{\lambda\in\Lambda}$ be a net in $\rho_{B}^{-1}\left( (-\infty, c]\right) $ such that $x_\lambda\to x$ in $(X, \tau)$. Then $\left(\frac{1}{c+\frac{1}{n}}\right)x_\lambda\in B$ for every $n\in\mathbb{N}$ and $\lambda\in\Lambda$. Note that for every $n\in\mathbb{N}$, $\left(\frac{1}{c+\frac{1}{n}}\right)x_\lambda\to\left(\frac{1}{c+\frac{1}{n}}\right)x$. Since $B$ is closed in $(X, \tau)$, $\left(\frac{1}{c+\frac{1}{n}}\right)x\in B$  for every $n\in \mathbb{N}$. Which implies that $\rho_B(x)\leq c+\frac{1}{n}$ for every $n\in\mathbb{N}$. Consequently, $\rho_B(x)\leq c$. Therefore $\rho_{B}^{-1}\left( (-\infty, c]\right) $ is closed in $(X, \tau)$. Hence $\rho_B$ is lower semi-continuous on $(X, \tau)$. By our assumption, $\rho_{B}$ is a continuous seminorm on $(X, \tau)$. Thus $B$ is a neighborhood of $0$ in $(X, \tau)$.
\end{proof}

\begin{corollary}\label{barreled implies mackey space} Let $(X,\tau)$ be a barreled elcs. If $\tau_{F}$ is the flc topology for $(X, \tau)$, then $(X,\tau_F) $ is a Mackey space.\end{corollary}

\begin{proof} Suppose $(X, \tau)$ is barreled. By Theorem \ref{barreledness of flc topology}, $(X, \tau_{F})$ is a barreled locally convex space. Hence $(X, \tau_{F})$ is a Mackey space as every barreled locally convex space is a Mackey space (see, Corollary 4.9, p. 98 in \cite{lcsosborne}).\end{proof}
	
\begin{corollary}\label{extended banach space implies mackey space} Let $(X,\parallel\cdot\parallel)$ be an extended Banach space. If $\tau_{F}$ is the flc topology for $(X,\parallel\cdot\parallel)$, then $(X,\tau_{F})$ is a Mackey space.\end{corollary}	

\begin{proof} It follows from Corollary \ref{barreledness of an extended Banach space} and Corollary \ref{barreled implies mackey space}. \end{proof}	
 
\begin{example} Suppose $X=Y=c_0$ is the collection of all real sequences converging to $0$. For $n\in\mathbb{N}$, consider the extended norm $\rho_n:X\rightarrow [0,\infty]$ defined by 	\[\rho_n((x^i))=
	\begin{cases}
		\text{$\infty$,} &\quad\text{ if $x^i\neq 0$  for some $1\leq i\leq n$}\\
		\text{$\displaystyle{\sup_{i\in\mathbb{N}}}|x^i|$,} &\quad\text{if $x^i= 0$ for every $1\leq i\leq n$.}\\
	\end{cases}
	\] Let $\tau$ be the topology induced by the collection $\mathcal{P}=\{\rho_n:n\in\mathbb{N}\}$. Then $(X, \tau)$ is an elcs and $X_{fin}=\{0\}$. By Example 5.12 in \cite{flctopology}, the flc topology $\tau_F$ for $(X, \tau)$ is induced by the norm $\parallel (x^i)\parallel_\infty=\sup_{i\in\mathbb{N}}|x^i|$. Note that $(X, \tau_F)$ is a Banach space. So by Theorem \ref{barreledness of flc topology}, $(X, \tau)$ is a barreled elcs.\end{example}

Now, we use Theorem \ref{barreledness of flc topology} to prove several permanence properties related to barreledness of an elcs.  

\begin{theorem} Let $M$ be a dense subspace of an elcs $(X, \tau)$. If $M$ is a barreled subspace of $(X, \tau)$, then $(X, \tau)$ is also barreled. \end{theorem}

\begin{proof}Suppose $\tau_F$ is the flc topology for $(X, \tau)$. Then by Theorem 4.1 in \cite{flctopology}, the flc topology for $(M, \tau|_M)$ is equal to $\tau_F|_M$. Therefore by Theorem \ref{barreledness of flc topology}, $(M, \tau_{F}|_M)$ is barreled. Since $M$ is dense in $(X, \tau)$, $M$ is a dense subspace of $(X, \tau_F)$ as $\tau_F\subseteq\tau$. By Theorem 11.2.1, p. 408 in \cite{tvsnarici}, $(X, \tau_F)$ is barreled. Hence by Theorem \ref{barreledness of flc topology}, $(X, \tau)$ is barreled.\end{proof}

\begin{theorem} Let $M$ be a subspace of a barreled elcs $(X, \tau)$ with countable co-dimension. Then $M$ is barreled.\end{theorem}

\begin{proof}By Theorem \ref{barreledness of flc topology}, $M$ is a subspace of the barreled locally convex space $(X, \tau_F)$ with countable co-dimension. Therefore by Theorem 11.2.9, p. 412 in \cite{tvsnarici}, $M$ is a barreled subspace of $(X, \tau_F)$. It is shown in Theorem 4.3 of  \cite{flctopology} that the flc topology for $(M, \tau|_M)$ is equal to $\tau_F|_M$. Hence by Theorem \ref{barreledness of flc topology}, $(M, \tau|_M)$ is barreled. \end{proof}

Our next result shows that if $\tau_F$ is the flc topology for an elcs $(X, \tau)$, then both $(X, \tau)$ and $(X, \tau_{F})$ have the same closed subspaces.

\begin{theorem}\label{closedness of a subspace in finest space}	Let $(X, \tau)$ be an elcs with the flc topology $\tau_{F}$. Suppose $Y$ is a subspace of $X$. Then $Y$ is closed in $(X, \tau)$ if and only if it is closed in $(X, \tau_F)$.\end{theorem}

\begin{proof}If $Y$ is closed in $(X, \tau_F)$, then $Y$ is closed in $(X, \tau)$ as $\tau_F\subseteq \tau$. Conversely, suppose $Y$ is closed in $(X, \tau)$ and $x_0\notin Y$. Let $Z= Y\oplus \text{ span }(x_0)$. Consider the linear functional $f$ on $Z$ defined by $f(y+\alpha x_0)=\alpha$ for every $y\in Y$ and $\alpha\in\mathbb{K}$. If $f$ is not continuous on $(Z, \tau|_Z)$, then there exists a net $(z_\lambda)_{\lambda\in\Lambda}$ in $Z$ such that $z_\lambda\to z$ in  $(Z, \tau|_Z)$ and $f(z_\lambda)\nrightarrow f(z)$. Let for every $\lambda$, $z_\lambda=y_\lambda+\alpha_\lambda x_0$ and $z= y+ \alpha x_0$. Then $\alpha_\lambda\nrightarrow \alpha$. Therefore there exists a $\delta>0$ such that for each $\lambda\in\Lambda$, there is a $\lambda'\geq\lambda$ satisfying $|\alpha_{\lambda'}-\alpha|>\delta$. Since $Y$ is closed in $(X, \tau)$ and $x_0\notin Y$, there exist a continuous extended seminorm $\rho$ on $X$ and $0<\epsilon<\delta$ such that $\rho(x_0-y)>\epsilon$ for every $y\in Y$. Also, we can find a $\lambda_0\in\Lambda$ such that   $\rho(y_{\lambda_0}+\alpha_{\lambda_0} x_0 - y- \alpha x_0)<\epsilon^2$ and $|\alpha_{\lambda_0}-\alpha|>\delta$. Since $\frac{y_{\lambda_0}-y}{\alpha-\alpha_{\lambda_0}}\in Y$, we have  $$\epsilon < \rho\left( \frac{y_{\lambda_0}-y}{\alpha-\alpha_{\lambda_0}}-x_0\right)\leq \frac{\rho\left(y_{\lambda_0}-y-(\alpha-\alpha_{\lambda_0})x_0\right)}{|\alpha_{\lambda_0}-\alpha|}< \frac{(\epsilon)^2}{\delta} <\epsilon.$$ We arrive at a contradiction. So $f$ is continuous on $(Z, \tau|_Z)$. By Corollary 4.2 in \cite{flctopology}, there exists a $g\in X^*$ such that $g|_Z=f$. Since both $(X, \tau)$ and $(X, \tau_F)$ have the same dual, $g^{-1}(0)$ is closed in $(X, \tau_F)$. Therefore $Y$ is closed in $(X, \tau_F)$ as $Y\subseteq g^{-1}(0)$ and $g(x_0)=1$.\end{proof}

Recall that if $M$ is a closed subspace of an elcs $(X, \tau)$, then the quotient space $X/M=\{x+M :x\in X\}$ with the quotient topology $\pi(\tau)$ is an elcs,  where $\pi(x)=x+M$ for $x \in X$ denotes the quotient map. Similarly, if $\{(X_i, \tau_i) : i\in\mathcal{I}\}$ is a collection of extended locally convex spaces, then the product space $\Pi_{i\in \mathcal{I}} X_i$ with the product topology $\Pi_{i\in \mathcal{I}}\tau_i$ is an elcs. For more details about quotient and product space of extended locally convex spaces, we refer to \cite{flctopology, esaetvs}. 

\begin{theorem} Let $(X, \tau)$ be a barreled elcs with the flc topology $\tau_F$. Suppose  $M$ is a closed subspace $X$. Then the quotient space $X/M$ with the quotient topology $\pi(\tau)$ is a barreled elcs. \end{theorem}

\begin{proof} By Theorems \ref{barreledness of flc topology} and \ref{closedness of a subspace in finest space}, $M$ is a closed subspace of the barreled locally convex space $(X, \tau_F)$. Then by Theorem 11.12.3, p. 409 in \cite{tvsnarici}, the quotient space $(X/M, \pi(\tau_F))$ is a barreled locally convex space. It is shown in Theorem 4.3 of  \cite{flctopology} that the flc topology for $\left( X/M, \pi(\tau)\right)$ is equal to $\pi(\tau_F)$. Therefore  by Theorem \ref{barreledness of flc topology},  $\left( X/M, \pi(\tau)\right)$ is barreled.  \end{proof}

\begin{theorem}Suppose $\{(X_i, \tau_i) : i\in\mathcal{I}\}$ is a collection of barreled extended locally convex spaces. Then the product space $\Pi_{i\in \mathcal{I}} X_i$ with the product topology $\Pi_{i\in \mathcal{I}}\tau_i$ is a barreled elcs.  \end{theorem}

\begin{proof} Let $\tau_{F_i}$ be the flc topology of the elcs $(X_i, \tau_i)$ for every $i\in\mathcal{I}$. Then by Theorem \ref{barreledness of flc topology}, $(X_i, \tau_{F_i})$ is a barreled locally convex space for every $i\in \mathcal{I}$. Therefore by Theorem 11.12.4, p. 409 in \cite{tvsnarici}, the product space  $\left(\Pi_{i\in \mathcal{I}}X_i, \Pi_{i\in \mathcal{I}}\tau_{F_i} \right)$  is a barreled locally convex space. It is shown in Theorem 4.4 of \cite{flctopology} that the flc topology for $\left(\Pi_{i\in \mathcal{I}}X_i, \Pi_{i\in \mathcal{I}}\tau_{i} \right)$ is equal to $\Pi_{i\in\mathcal{I}}\tau_{F_i}$. Hence by Theorem \ref{barreledness of flc topology}, $\left(\Pi_{i\in \mathcal{I}}X_i, \Pi_{i\in \mathcal{I}}\tau_{i} \right)$ is a barreled elcs.\end{proof}

\section{Function Spaces}
For a metric space $(X,d)$, let $C(X)$ denote the space of real-valued continuous functions on  $(X,d)$. A collection $\mathcal{B}$ of nonempty sets in a metric space $(X, d)$ is said to be a \textit{bornology} if it covers $X$ and is closed under finite union and taking subsets of its members. A \textit{base} $\mathcal{B}_0$ is a subfamily of $\mathcal{B}$ which is cofinal in $\mathcal{B}$ under the set inclusion. If every member of $\mathcal{B}_0$ is closed in $(X,d)$, then $\mathcal{B}$ is said to have a \textit{closed base}. For details on metric bornologies, we refer to the monograph \cite{bala}.  

 In this section,  we study the barreledness of the spaces $(C(X),\tau_{\mathcal{B}})$ and $(C(X),\tau_{\mathcal{B}}^{s})$, where $\tau_{\mathcal{B}}$ denotes the topology of uniform convergence on $\mathcal{B}$, and $\tau_{\mathcal{B}}^{s}$ denotes the topology of strong uniform convergence on $\mathcal{B}$. 


\begin{definition} \normalfont Let $\mathcal{B}$ be a bornology on a metric space $(X, d)$. Then \textit{the topology $\tau_{\mathcal{B}}$ of  uniform convergence on $\mathcal{B}$} is determined by a uniformity on $C(X)$ having base consisting of sets of the form  $$[B, \epsilon]=\left\lbrace (f, g) : \forall x\in B,~ |f(x)-g(x)|<\epsilon \right\rbrace~ (B\in\mathcal{B},~ \epsilon>0). $$  \end{definition}

\begin{definition}{\normalfont (\cite{Suc})} \normalfont Let $\mathcal{B}$ be a bornology on a metric space $(X, d)$. Then \textit{the topology $\tau^{s}_{\mathcal{B}}$ of strong uniform convergence on $\mathcal{B}$} is determined by a uniformity on $C(X)$ having base consisting of sets of the form  $$[B, \epsilon]^s=\left\lbrace (f, g) : \exists~ \delta>0 ~\forall x\in B^\delta,~ |f(x)-g(x)|<\epsilon \right\rbrace~ (B\in\mathcal{B},~ \epsilon>0),$$ where for $B\subseteq X$, $B^\delta=\displaystyle{\bigcup_{y\in B}}\{x\in X: d(x, y)<\delta\}$.  	 \end{definition}

The topology $\tau_{\mathcal{B}}$ on $C(X)$ can also be induced by the collection $\mathcal{P}=\{\rho_{B}: B\in\mathcal{B}\}$ of extended seminorms, where $\rho_{B}(f)=\sup_{x\in B}|f(x)|$ for $f\in C(X)$. Similarly, the topology $\tau_{\mathcal{B}}^s$ on $C(X)$ is induced by the collection $\mathcal{P}=\{\rho_{B}^s: B\in\mathcal{B}\}$ of extended seminorms, where $\rho_B^s(f) =\inf_{\delta>0}\left\lbrace \sup_{x\in B^\delta}|f(x)|\right\rbrace$ for $f\in C(X)$. Therefore both $(C(X), \tau_{\mathcal{B}})$ and $(C(X), \tau_{\mathcal{B}}^s)$ are extended locally convex spaces. When $\mathcal{B}$ is the bornology of all finite subsets and we extend our function space to $\mathbb{R}^X$, the set of all real-valued functions on $(X,d)$, $\tau_{\mathcal{B}}^s$ is the weakest topology on $\mathbb{R}^X$ for which $C(X)$ is closed (\cite{Suc}, Corollary 6.8).  For more details related to $\tau_{\mathcal{B}}$ and $\tau_{\mathcal{B}}^s$, we refer to \cite{Suc, ATasucob}.

\begin{theorem}\label{barreledness on tauBs} Let $\mathcal{B}$ be a bornology with a closed base on a metric space $(X, d)$. If $(C(X), \tau_{\mathcal{B}}^s)$ is barreled, then the bornology $\mathcal{K}$ of all non-empty relatively compact subsets of $X$ is contained in $\mathcal{B}$. \end{theorem}

\begin{proof}	Let $K$ be a compact subset of $X$. Consider  $V_K=\{f\in C(X): \sup_{x\in K}|f(x)|\leq 1\}$. Then $V_K$ is an absolutely convex and absorbing subset of $C(X)$. If $(f_\lambda)_{\lambda\in \Lambda}$ is a net in $V_K$ converging to $f$ in $(C(X), \tau_{\mathcal{B}}^s)$, then $f_\lambda(x)\to f(x)$ for every $x\in X$. For every $\epsilon>0$ and $x\in K$, there exists an $\lambda_o\in \Lambda$ such that $|f_{\lambda_0}(x)-f(x)|<\epsilon$. So $|f(x)|\leq 1$ as $|f(x)|\leq |f(x)-f_{\lambda}(x)|+|f_\lambda(x)|$ for every $\lambda\in\Lambda$. Consequently, $f\in V_K$. Therefore $V_K$ is closed in $(C(X), \tau_{\mathcal{B}}^s)$. Hence  $V_K$ is a barrel in $(C(X), \tau_{\mathcal{B}}^s)$. Since $(C(X), \tau_{\mathcal{B}}^s)$ is barreled,  $V_K$ is a neighborhood of $0$ in $C(X, \tau_{\mathcal{B}}^s)$. Therefore  $ U= \{f\in C(X): \rho_{B}^{s}(f)<\epsilon\}\subseteq V_K$ for some closed set $B \in \mathcal{B}$ and $\epsilon>0$. It is sufficient to show that $K\subseteq B$. Suppose $x_0\in K\setminus B$. Then there exists an $m\in\mathbb{N}$ such that $d(x_0, y)> \frac{1}{m}$ for every $y\in B$ as $B=\bigcap_{n\in\mathbb{N}} \{ x\in X : d(x, y)\leq \frac{1}{n} \text{ for some } y\in B\}$. It is easy to see that $\text{Cl}_X(B^{\frac{1}{2m}})\subseteq \{ x\in X : d(x, y)\leq \frac{1}{m} \text{ for some } y\in B\}$. So by Urysohn's lemma, there exists a continuous function $f$ on $X$ such that $f|_{\text{Cl}_X(B^{\frac{1}{2m}})}=0$ and $f(x_0)=2$. Consequently, $f\in U\setminus V_K$. We arrive at  a contradiction.  Therefore $K\subseteq B$. Hence  $K\in \mathcal{B}$.\end{proof}	

\begin{theorem}\label{barreledness of tauB} Let $\mathcal{B}$ be a bornology with a closed base on a metric space $(X, d)$. If $(C(X), \tau_{\mathcal{B}})$ is barreled, then the bornology $\mathcal{K}$ of all non-empty relatively compact subsets of $X$ is contained in $\mathcal{B}$.\end{theorem}	

\begin{proof} It is similar to the proof of Theorem \ref{barreledness on tauBs}.\end{proof}	

\begin{corollary}\label{compact open topology coarser than  the finest topology} Let $\mathcal{B}$ be a bornology with closed base on a metric space $(X, d)$. Suppose  $\tau_F$ and $\tau_{F}^s$ are the flc topologies for the spaces $(C(X), \tau_{\mathcal{B}})$ and $(C(X), \tau_{\mathcal{B}}^s)$, respectively. If $(C(X), \tau_{\mathcal{B}})$ and $(C(X), \tau_{\mathcal{B}}^s)$ are barreled, then the compact open topology $\tau_{\mathcal{K}}$ is coarser than both $\tau_{F}$ and $\tau_{F}^s$. \end{corollary}	

The next example shows that Theorem \ref{barreledness on tauBs} and Theorem \ref{barreledness of tauB} may not hold if $\mathcal{B}$ does not have any closed base in $X$.

\begin{example}\label{counter example for K not subset of B} Let $X=\mathbb{R}$ with the usual metric, and let $\mathcal{B}=\{ A\subseteq \mathbb{R}: A \text{ is a non-empty subset of } \mathbb{Q}\cup F \text{ for a finite subset } F \text{ of  } \mathbb{R}\}.$ Then $\mathcal{B}$ forms a bornology on $X$ without any closed base. Note that $A=[0, 1]$ is  a compact set in $X$ and $A\notin \mathcal{B}$. Also, it is easy to see that $\tau_{\mathcal{B}}=\tau_{\mathcal{B}}^s=\tau_u$, where $\tau_u$ is the topology of uniform convergence on $C(X)$ induced by the extended norm $\parallel f\parallel_{\infty}=\sup_{x\in X} |f(x)|$ for $f\in C(X)$. By Corollary \ref{barreledness of a uniform space}, $(C(X), \tau_{u})$ is barreled.\end{example}

The next example shows that if $\mathcal{B}$ is a bornology with a closed base on a metric space $(X, d)$, and both $(C(X), \tau_{\mathcal{B}})$ and $(C(X), \tau_{\mathcal{B}}^s)$ are barreled, then the inclusions in Corollary \ref{compact open topology coarser than  the finest topology} may be strict, that is, if $\mathcal{B}$ is a bornology with a closed base on a metric space $(X, d)$, then the compact open topology $\tau_{\mathcal{K}}$ may not be equal to the flc topologies $\tau_F$ and $\tau_F^s$ for the spaces $(C(X), \tau_{\mathcal{B}})$ and $(C(X), \tau_{\mathcal{B}}^s)$, respectively.

\begin{example}\label{example for strict inculsion of compact open topology and flc topology}
	Let $(X, d)$ be a discrete metric space	such that $X=\{x_n:n\in\mathbb{N}\}$. If $\mathcal{B}$ is the collection of all non-empty subsets of $X$, then both $\tau_{\mathcal{B}}$ and $\tau_{\mathcal{B}}^s$ are induced by the extended norm $\parallel f\parallel_{\infty}=\sup_{x\in X} |f(x)|$ for $f\in C(X)$. By Corollary \ref{barreledness of a uniform space},  $(C(X), \parallel\cdot\parallel_{\infty})$ is barreled. Let $\tau_F$ be the flc topology for $(C(X), \parallel\cdot\parallel_{\infty})$. We show that $\tau_{\mathcal{K}}\neq \tau_F$. For  $m\in\mathbb{N}$, define a function $f_m$ on $X$ by 
	\[f_m(x_n)=\begin{cases}
		\text{$1$,} &\quad\text{ if $m=n$ }\\
		\text{0,} &\quad\text{ if $m\neq n$.}\\
	\end{cases}\] 
	Since $(X, d)$ is a discrete space, $f_m\in C(X)$ for every $m\in\mathbb{N}$. Also, it is easy to see that $f_m\to 0$ in $(C(X), \tau_{\mathcal{K}})$. Let $M= \text{span}(\{f_m:m\in\mathbb{N}\})$. Then $(M, \parallel\cdot\parallel_\infty)$ is a normed linear space. Since $\parallel f_m\parallel_\infty=1$ for every $m\in\mathbb{N}$, $f_m\nrightarrow 0$ in $(M, \parallel\cdot\parallel_\infty)$. Thus $f_m\nrightarrow 0$ in $(C(X), \tau_F)$ as $\tau_F|_M=\tau_{\parallel\cdot\parallel_\infty}$ on $M$.
\end{example}

\section{uniform boundedness principle}	

In this section, we establish the uniform boundedness principle for a barreled extended locally convex space. In order to do that we use the barreledness of the corresponding finest space. We first relate the continuity of a linear operator between extended locally convex spaces to the continuity of the same under the corresponding finest spaces. Then we extend this result to a equicontinuous family of continuous linear operators.

The following result follows from Proposition 4.8 of \cite{esaetvs}.
\begin{proposition}\label{continuity conditions for a linear operator} Let $(X, \tau), (Y, \sigma)$ be two extended locally convex spaces. Suppose $T:X\rightarrow Y$ is a linear map. Then $T$ is continuous if and only if for every continuous extended seminorm $q$ on $Y$, there exists a continuous extended seminorm $p$ on $X$ such that $q(T(x))\leq p(x)$ for every $x\in X$. \end{proposition}

\begin{corollary}
Let $(X, \tau), (Y, \sigma)$ be two extended locally convex spaces. Suppose $T:X\rightarrow Y$ is a continuous linear map. Then for every continuous extended seminorm $q$ on $(Y, \sigma)$, there exists a continuous extended seminorm $p$ on $(X, \tau)$ such that $T^{-1}\left( Y_{fin}^q\right)\supseteq X_{fin}^p.$
\end{corollary}	

\begin{theorem}\label{continuity of linear operator in finest and extended space} Let $(X, \tau), (Y, \sigma)$ be two extended locally convex spaces. Suppose $T:X\rightarrow Y$ is a linear map such that for every continuous extended seminorm $q$ on $(Y, \sigma)$, there exists a continuous extended seminorm $p$ on $(X, \tau)$ such that $$T^{-1}\left( Y_{fin}^q\right)\supseteq X_{fin}^p.$$ Then the following statements are equivalent.
\begin{itemize}
\item[(1)] $T:(X, \tau)\rightarrow (Y,\sigma)$ is continuous.
\item[(2)] $T:(X, \tau{_{F}})\rightarrow (Y,\sigma{_{F}})$ is continuous, where $\tau_{F}$ and $\sigma_{F}$ are the flc topologies for the spaces $(X, \tau)$ and $(Y, \sigma)$, respectively.\end{itemize}
\end{theorem}	

\begin{proof}$(1)\Rightarrow(2)$. Suppose $q$ is a continuous seminorm on $(Y, \sigma_{F})$. Then $q$ is a continuous with respect to $(Y, \sigma)$. Since  $T:(X, \tau)\rightarrow (Y,\sigma)$ is continuous, by Proposition \ref{continuity conditions for a linear operator}, there exists a continuous extended seminorm $p$ on $(X, \tau)$ such that $q(T(x))\leq p(x)$ for all $x\in X$. Let $M$ be a subspace of $X$ such that $X=X_{fin}^p\oplus M$. Consider the seminorm	$\rho:X\rightarrow [0,\infty)$ defined by $$\rho(x)=p(x_f)+q(T(x_m)),$$ where $x=x_f+x_m$ for $x_f\in X_{fin}^p$ and $x_m\in M.$	 Since $\rho(x)\leq p(x)$ for every $x\in X$, $\rho$ is a continuous seminorm on $(X,\tau)$. By Theorem 3.5 in \cite{flctopology}, $\rho$ is continuous on $(X, \tau_{F})$. It is easy to prove that $q(T(x))\leq \rho(x)$ for $x\in X$. Therefore by Theorem 5.7.3,  p. 126 in \cite{tvsnarici}, $T:(X, \tau{_{F}})\rightarrow (Y,\sigma{_{F}})$ is continuous. 
	
$(2)\Rightarrow (1)$. Suppose $q$ is a continuous extended seminorm on $(Y, \sigma)$. Then by our assumptions, there exists a continuous extended seminorm $p$ on $(X, \tau)$ such that  $T^{-1}\left(Y_{fin}^q \right)\supseteq X_{fin}^{p}$. Let $N$ be a subspace of $Y$ such that $Y=Y_{fin}^q\oplus N$. Define a seminorm $\mu$ on $Y$ by $\mu(y)=q(y_f)$, where $y=y_f+y_n$ for $y_f\in Y_{fin}^q$  and $y_n\in N$. Since $\mu(y)\leq q(y)$ for every $y\in Y$, $\mu$ is a continuous seminorm on $(Y, \sigma)$. By  Theorem 3.5 in \cite{flctopology}, $\mu$ is continuous on $(Y, \sigma_{F})$. Since $T:(X, \tau{_{F}})\rightarrow (Y,\sigma_{F})$ is continuous, by Theorem 5.7.3, p. 126 in \cite{tvsnarici}, there exists a continuous seminorm $\rho$ on $(X, \tau_{F})$ such that $\mu(T(x))\leq \rho(x)$ for every $x\in X$. It is easy to see that $p+\rho$ is a continuous extended seminorm on $(X, \tau)$. Note that if $q(T(x))=\infty$, then $p(x)=\infty$. Otherwise, $q(T(x))=\mu(T(x))\leq \rho(x)$. Therefore $$q(T(x))\leq (p+\rho)(x) \text{ for every } x\in X.$$ 
Hence $T:(X, \tau)\rightarrow (Y,\sigma)$ is continuous.	\end{proof}

\begin{corollary}Let $(X, \tau), (Y, \sigma)$ be two barreled extended locally convex spaces. Suppose $T:X\rightarrow Y$ is a linear map such that for every continuous extended seminorm $q$ on $Y$, there exists a continuous extended seminorm $p$ on $X$ such that $$T^{-1}\left( Y_{fin}^q\right)\supseteq X_{fin}^p.$$ Then the following statements are equivalent.
\begin{itemize}
\item[(1)] $T:(X, \tau)\rightarrow (Y,\sigma)$ is continuous.
\item[(2)] $T:(X, \tau{_{w}})\rightarrow (Y,\sigma{_{w}})$ is continuous, where $\tau_{w}$ and $\sigma_{w}$ are the weak topologies for the spaces $(X, \tau)$ and $(Y, \sigma)$, respectively.\end{itemize}\end{corollary}

\begin{proof}It follows from Corollary 11.3.7, p. 385 in \cite{tvsnarici}, Theorem \ref{barreledness of flc topology}, and Theorem \ref{continuity of linear operator in finest and extended space}.\end{proof}

In general, we may not drop the condition $T^{-1}\left( Y_{fin}^q\right) \supseteq X_{fin}^p$ in Theorem \ref{continuity of linear operator in finest and extended space}.

\begin{example}	Let $X=Y=\mathbb{R}^2$. If $\tau$ is the topology induced by  $\parallel(x, y)\parallel=|x|+|y|$ for $(x, y)\in X$, and $\sigma $ is the topology induced by the discrete extended norm, then $X_{fin}=\mathbb{R}^2$ in $(X, \tau)$ and $Y_{fin}=\{(0, 0)\}$ in $(Y, \sigma)$. Note that $\tau_{F}=\sigma_{F}$ as the dimension of $X$ is finite. Then the identity map $I:(X,\tau_{F})\rightarrow (Y,\sigma_{F})$ is continuous but $I:(X,\tau)\rightarrow (Y,\sigma) $ is not continuous.  \end{example}
 
 To define a pointwise bounded family of continuous linear operators, we recall the definition of a bounded set in an elcs and some of its basic properties from \cite{flctopology, doelcs}.	
\begin{definition}\label{bounded set in an elcs}\normalfont We say a subset $A$  of an elcs $(X, \tau)$ is bounded if for every neighborhood $U$ of $0$, there exist a finite subset $F$ of $X$ and $r>0$ such that $A\subseteq F+rU$. \end{definition}

\begin{proposition} Suppose $(X, \tau)$ is an elcs. Then the following statements hold.
	\begin{itemize}
		\item[(1)] The collection of all bounded sets in an elcs forms a bornology. 
		\item[(2)] $A$ is bounded in $(X, \tau)$ if and only if for every neighborhood $U$ of $0$, there exist a finite subset $F$ of $A$ and $r>0$ such that $A\subseteq F+rU$.
		\item[(3)] Suppose $(Y,\sigma)$ is an elcs and $T:(X,\tau)\rightarrow (Y,\sigma)$ is a  continuous linear operator. Then for each bounded set $A$ in $(X,\tau)$, $T(A)$ is bounded in $(Y,\sigma)$.
		\item[(4)] No subspace $\left( \text{except } \{0 \}\right)$ of $X$ is bounded.
		\item[(5)] If $(x_n)$ converges to $0$ in $(X,\tau)$, then $\{x_n:n\in \mathbb{N}\}$ is  bounded in  $(X, \tau)$.
	    \item[(6)]  For a conventional locally convex space $(X, \tau)$, $A\subseteq X$ is bounded in the sense of Definition \ref{bounded set in an elcs} if and only if it is absorbed by each neighborhood of $0$ in $(X, \tau)$. \end{itemize}
\end{proposition}

Let $\mathcal{L}(X,Y)$ denote the collection of all continuous linear maps from an elcs $(X, \tau)$ to an elcs $(Y, \sigma)$.  Then $\mathcal{A}\subseteq \mathcal{L}(X,Y)$ is said to be \textit{pointwise bounded} if for every $x\in X$, $\{Tx: T\in\mathcal{A}\}$ is bounded in $(Y, \sigma)$. 

\begin{definition}\label{equicontinuous family} \normalfont Suppose $(X, \tau)$ and $(Y, \sigma)$ are extended locally convex spaces. Then $\mathcal{A}\subseteq \mathcal{L}(X, Y)$ is said to be \textit{equicontinuous} in $\mathcal{L}(X, Y)$ if for every neighborhood $V$ of $0$ in $(Y, \sigma)$, there exists a neighborhood $U$ of $0$ in $(X, \tau)$ such that $T(U)\subseteq V \text{ for } T\in\mathcal{A}$. \end{definition}

\noindent The following facts regarding equicontinuity of a subset $A$ of $X^*$ are easy to verify (\cite{doelcs}). 
\begin{itemize}
	\item[(1)] $A$ is equicontinuous on $(X, \tau)$ if and only if $A_\circ=\{x\in X: |f(x)|\leq 1 \text{ for every } f\in A\}$ is a neighborhood of $0$ in $(X, \tau)$.
	\item[(2)] If $A\subseteq B\subseteq X^*$ and $B$ is equicontinuous, then $A$ is also equicontinuous.
	\item[(3)] If  $(X, \parallel\cdot\parallel)$ is an enls, then  $A\subseteq X^*$ is equicontinuous if and only if there exists an $M>0$ such that $\parallel f\parallel_{op}\leq M$ for each $f\in A$, where $$\parallel f\parallel_{op}= \sup\{| f(x)|: \parallel x\parallel\leq 1\}.$$  In particular, the closed unit ball $B_{X^*}=\{f\in X^*: \parallel f\parallel_{op}\leq 1\}$ is always equicontinuous on $X$.\end{itemize}  

\begin{remark} If  $(X, \parallel\cdot\parallel)$ is an enls and $x\in X\setminus X_{fin}$, then for any $n\in\mathbb{N}$ there exists a linear functional $f_n\in B_{X^*}$ such that $f_n(x)=n$. So $B_{X^*}$ is not pointwise bounded. Hence in contrast to classical locally convex spaces, an equicontinuous family on an elcs may not be pointwise bounded.
\end{remark}


\begin{theorem}\label{exteneded seminorm crtiterion for equicontinuous family} Suppose $(X, \tau)$ and $(Y, \sigma)$ are extended locally convex spaces. Then $\mathcal{A}\subseteq \mathcal{L}(X, Y)$ is equicontinuous if and only if for every continuous extended seminorm $q$ on $Y$, there exists a continuous extended seminorm $p$ on $X$ such that $q(T(x))\leq p(x) \text{ for every } T\in\mathcal{A} \text{ and } x\in X$.   \end{theorem}

\begin{proof} Suppose $\mathcal{A}\subseteq \mathcal{L}(X, Y)$ is equicontinuous and $q$ is any continuous extended seminorm on $(Y, \sigma)$. Then there exists an absolutely convex neighborhood $U$ of $0$ in $(X, \tau)$ such that $T(U)\subseteq q^{-1}\left( [0,1]\right) $ for every $T\in\mathcal{A}$. Let $p=\rho_U$, where $\rho_{U}$ is the Minkowski functional on $X$ corresponding to $U$. Then $\rho_U$ is a continuous extended seminorm on $(X, \tau)$. It is easy to see that if $p(x)<\infty$, then $$\frac{x}{p(x)+\epsilon}\in p^{-1}\left( [0, 1)\right) \subseteq U\subseteq T^{-1}\left( q^{-1}\left( [0,1]\right) \right) $$ for every $\epsilon>0$.   Therefore $q(T(x))\leq p(x)$ for every $x\in X$ and $T\in\mathcal{A}$.       
	
Conversely, suppose $\mathcal{A}\subseteq \mathcal{L}(X, Y)$ has the given property. If $V$ is a neighborhood of $0$ in $(Y, \sigma)$, then there exists a continuous extended seminorm $q$ on  $(Y, \sigma)$ such that $q^{-1}([0, 1))\subseteq V$. By our assumption, there exists a continuous extended seminorm $p$ on $(X, \tau)$ such that $q(T(x))\leq p(x)$ for every $x\in X$ and $T\in\mathcal{A}$. It is easy to see that $U=p^{-1}([0, 1))$ is a neighborhood of $0$ in $(X, \tau)$ and  that $T(U)\subseteq V$ for every $T\in\mathcal{A}$. \end{proof}

\begin{corollary}\label{equicontinuity condition in finite subspace} Let $(X, \tau)$ and $(Y, \sigma)$ be two extended locally convex spaces. Suppose $\mathcal{A}\subseteq \mathcal{L}(X, Y)$ is equicontinuous. Then for every continuous extended seminorm $q$ on $(Y, \sigma)$, there exists a continuous extended seminorm $p$ on $(X, \tau)$ such that $$\bigcap_{T\in \mathcal{A}}T^{-1}\left( Y_{fin}^q\right)\supseteq X_{fin}^p.$$\end{corollary}


\begin{theorem}\label{Uniform bounded principle}{\normalfont\textbf{(Uniform Boundedness Principle)}} Let $(X, \tau)$ be a barreled elcs, and let $(Y, \sigma)$ be any elcs. Suppose $\mathcal{A}\subseteq \mathcal{L}(X,Y)$ is pointwise bounded and for every continuous extended seminorm $q$ on $(Y, \sigma)$, there exists a continuous extended seminorm  $p$ on  $(X, \tau)$ such that $T^{-1}\left( Y_{fin}^q\right)\supseteq X_{fin}^p  \text{ for all } T\in \mathcal{A}$. Then $\mathcal{A}$ is an equicontinuous family in $\mathcal{L}(X, Y)$.\end{theorem}

\begin{proof} Let $\mathcal{L}(X^F, Y^F)$ be the collection of all continuous linear maps from $(X, \tau_{F})$ to $(Y, \sigma_{F})$, where $\tau_{F}$ and $\sigma_{F}$ are the flc topologies for $(X, \tau)$ and $(Y, \sigma)$, respectively. By Theorem \ref{continuity of linear operator in finest and extended space}, $\mathcal{A}\subseteq \mathcal{L}(X^F,Y^F)$. Therefore $\mathcal{A}$ is a pointwise bounded subset of $\mathcal{L}(X^F,Y^F)$ as $\sigma_{F}\subseteq \sigma$. Since $(X, \tau)$ is barreled, by Theorem \ref{barreledness of flc topology}, $(X, \tau_{F})$ is barreled. Thus by Theorem 11.9.1, p. 400 in \cite{tvsnarici}, $\mathcal{A}$ is an equicontinuous family in $\mathcal{L}(X^F, Y^F)$. Suppose $q$ is a continuous extended seminorm on $(Y, \sigma)$. Then there exists a subspace $N$ of $Y$ such that $Y= Y_{fin}^q\oplus N$. Consider the seminorm $q_1$ defined by $$q_1(y=y_f+y_n)=q(y_f)$$ for $y_f\in Y_{fin}^q$ and $y_n\in N$. Since $q_1(y)\leq q(y)$ for every $y\in Y$, $q_1$ is a continuous seminorm on $(Y, \sigma)$. By Theorem 3.5 in \cite{flctopology}, $q_1$ is a continuous seminorm on $(Y, \sigma_{F})$. Since $\mathcal{A}$ is an equicontinuous family in $\mathcal{L}(X^F, Y^F)$,  there exists a continuous seminorm $p_1$ on $(X, \tau_{F})$ such that $q_1(T(x))\leq p_1(x)$ for every $T\in \mathcal{A} \text{ and } x\in X$ (see, Theorem 8.6.3, p. 245 in \cite{tvsnarici}). Also, by our assumption, there exists a continuous extended seminorm $p$ on $(X, \tau)$ such that $T^{-1}\left( Y_{fin}^q\right)\supseteq X_{fin}^p  \text{ for all } T\in \mathcal{A}$. It is easy to see that $\mu(x)=p(x)+p_1(x)$ for $x\in X$ is a continuous extended seminorm on $(X, \tau)$ and that $q(T(x))\leq \mu(x)$ for every $x\in X$. Hence by Theorem \ref{exteneded seminorm crtiterion for equicontinuous family}, $\mathcal{A}$ is an equicontinuous family in $\mathcal{L}(X, Y)$.    	 \end{proof}

\begin{corollary}\label{equicontinuity in fundamental elcs} Let $(X, \tau)$ be a barreled fundamental elcs, and let $(Y, \sigma)$ be an elcs. Suppose $\mathcal{A}\subseteq \mathcal{L}(X,Y)$ is pointwise bounded. Then $\mathcal{A}$ is an equicontinuous family in $\mathcal{L}(X, Y)$.
\end{corollary}	
\begin{proof} Suppose $p$ is a continuous extended seminorm on $(X, \tau)$ such that $X_{fin}^p=X_{fin}$. Note that $X_{fin}$ and $Y_{fin}$ are the connected components of zero in $(X, \tau)$ and $(Y, \sigma)$, respectively. Therefore $T(X_{fin})\subseteq Y_{fin}$ for every $T\in \mathcal{A}$. Since $Y_{fin}\subseteq Y_{fin}^q$ for every continuous extended seminorm $q$ on $(Y, \sigma)$, we have $$X_{fin}^p\subseteq T^{-1}\left( Y_{fin}^q\right)$$ for every continuous extended seminorm $q$ on $(Y, \sigma)$. Hence by  Theorem \ref{Uniform bounded principle},  $\mathcal{A}$ is an equicontinuous family in $\mathcal{L}(X, Y)$.
\end{proof}

\begin{corollary}Let $(X, \parallel\cdot\parallel)$ be an extended Banach space, and let $(Y, \sigma)$ be an elcs. Suppose $\mathcal{A}\subseteq \mathcal{L}(X,Y)$ is pointwise bounded. Then $\mathcal{A}$ is an equicontinuous family in $\mathcal{L}(X, Y)$. Moreover, there exists an $M>0$ such that $\parallel T\parallel_{op}\leq M$ for every $T\in \mathcal{A}$. 
\end{corollary}	

\begin{theorem} Suppose $(X,\tau)$ is an elcs with flc topology $\tau_F$. Then $(X,\tau)$ is barreled if and only if every pointwise bounded subset $\mathcal{A}$ of $X^*$ is equicontinuous on $(X,\tau)$.\end{theorem}

\begin{proof} Suppose $(X, \tau)$ is barreled. Then by Theorem \ref{Uniform bounded principle}, every pointwise bounded subset $\mathcal{A}$ of $X^*$ is equicontinuous on $(X,\tau)$.
	
	Conversely, let $B$ be a barrel in $(X, \tau_F)$. Then $B^\circ= \{f\in X^*: |f(x)|\leq 1 \text{ for } x\in B\}$ is pointwise bounded. By our assumptions, $B^\circ$ is equicontinuous on $(X, \tau)$. Therefore $(B^\circ)_\circ=\{x\in X: |f(x)|\leq 1 \text{ for} f\in B^\circ\}$ is a neighborhood of $0$ in $(X, \tau)$. Since $B$ is a barrel in $(X, \tau_F)$, by Bipolar Theorem (Theorem 8.3.8, p. 235 in \cite{tvsnarici}), $B=(B^\circ)_\circ$. Which implies that $B$ is a neighborhood of $0$ in  $(X, \tau)$. Consequently, the Minkowski functional $\rho_B$ is a continuous seminorm on $(X, \tau)$. By Theorem 3.4 in \cite{flctopology}, $\rho_B$ is a continuous on $(X, \tau_F)$. Therefore $B$ is a neighborhood of $0$ in $(X, \tau_F)$. Hence $(X, \tau_F)$ is barreled. Consequently, by Theorem \ref{barreledness of flc topology}, $(X, \tau)$ is barreled. \end{proof}

\bibliographystyle{plain}
\bibliography{reference_file}		

\def\cprime{$'$} \def\cprime{$'$} \def\cprime{$'$}
\begin{thebibliography}{10}

\bibitem{nwiv}
G.~Beer.
\newblock Norms with infinite values.
\newblock {\em Journal of Convex Analysis}, 22(1):37--60, 2015.

\bibitem{bala}
G.~Beer.
\newblock {\em Bornologies and Lipschitz Analysis}.
\newblock CRC Press, Boca Raton, Florida, 2023.

\bibitem{Suc}
G.~Beer and S.~Levi.
\newblock Strong uniform continuity.
\newblock {\em Journal of Mathematical Analysis and Applications},
  350(2):568--589, 2009.

\bibitem{socsiens}
G.~Beer and J.~Vanderwerff.
\newblock Separation of convex sets in extended normed spaces.
\newblock {\em Journal of the Australian Mathematical Society}, 99(2):145--165,
  2015.

\bibitem{spoens}
G.~Beer and J.~Vanderwerff.
\newblock Structural properties of extended normed spaces.
\newblock {\em Set-Valued and Variational Analysis}, 23(4):613--630, 2015.

\bibitem{ATasucob}
A.~Caserta, G.~Di~Maio, and L.~Hol{\'a}.
\newblock Arzel{\`a}'s {T}heorem and strong uniform convergence on bornologies.
\newblock {\em Journal of Mathematical Analysis and Applications},
  371(1):384--392, 2010.

\bibitem{doelcs}
A.~Kumar and V.~Jindal.
\newblock Dual of an extended locally convex space.
\newblock {\em arXiv preprint arXiv:2301.03243}, 2023.

\bibitem{flctopology}
A.~Kumar and V.~Jindal.
\newblock The finest locally convex topology of an extended locally convex
  space.
\newblock {\em Topology and its Applications}, 326:108396, 2023.

\bibitem{tvsnarici}
L.~Narici and E.~Beckenstein.
\newblock {\em Topological vector spaces}.
\newblock Second Edition, CRC Press, 2011.

\bibitem{lcsosborne}
M.~S. Osborne.
\newblock {\em Locally convex spaces}.
\newblock Springer, 2014.

\bibitem{esaetvs}
D.~Salas and S.~Tapia-Garc{\'\i}a.
\newblock Extended seminorms and extended topological vector spaces.
\newblock {\em Topology and its Applications}, 210:317--354, 2016.

\bibitem{tvsschaefer}
H.~H. Schaefer and M.~P. Wolff.
\newblock {\em Topological vector spaces}.
\newblock Second Edition, Springer-Verlag, New York, 1999.

\bibitem{willard}
S.~Willard.
\newblock {\em General topology}.
\newblock Addison-Wesley Publishing Co., Reading, Mass.-London-Don Mills, Ont.,
  1970.

\end{thebibliography}
\end{document}